\documentclass{amsart}
\usepackage[margin=1.1in]{geometry}
\usepackage{amsthm,amsmath,amssymb,enumitem,mathtools,mathrsfs,graphicx,url}
\usepackage{xcolor}
\usepackage[colorlinks=true,linkcolor=blue,citecolor=blue,urlcolor=blue,bookmarks=false]{hyperref}
\usepackage{tikz}
\usepackage{longtable}
\usetikzlibrary{arrows.meta}
\numberwithin{equation}{section}

\theoremstyle{plain}
\newtheorem{thm}{Theorem}[section]
\newtheorem{cor}[thm]{Corollary}
\newtheorem{prop}[thm]{Proposition}
\newtheorem{lem}[thm]{Lemma}
\newtheorem{claim}[thm]{Claim}
\theoremstyle{definition}

\newtheorem{example}[thm]{Example}

\theoremstyle{remark}
\newtheorem{rmk}[thm]{Remark}

\DeclareMathOperator{\codeg}{codeg}
\newcommand{\ZZ}{\mathbb Z}
\newcommand{\RR}{\mathbb R}
\newcommand{\ab}{\mathbf a}
\providecommand{\Fc}{}
\renewcommand{\Fc}{\mathcal F}
\providecommand{\Oc}{}
\renewcommand{\Oc}{\mathcal O}

\providecommand{\Dc}{}
\renewcommand{\Dc}{\mathcal D}
\newcommand{\oplu}{\oplus}

\DeclareMathOperator{\aff}{aff}

\tikzset{
  fvertex/.style={circle, draw, fill=white, inner sep=2pt},
  fedge/.style={-{Stealth[length=2mm]}, thick},
  flabel/.style={draw=none, fill=none, rectangle, inner sep=1pt,
                 font=\small\bfseries}
}

\title[A matroidal criterion for flow polytopes to be order polytopes]{A matroidal criterion for flow polytopes to be order polytopes}
\author{Akihiro Higashitani}
\address{Graduate School of Information Science and Technology, Osaka University, Suita, Osaka 565-0871, Japan}
\email{higashitani@ist.osaka-u.ac.jp}
\author{Hidefumi Ohsugi}
\address{Department of Mathematical Sciences, School of Science, Kwansei Gakuin University, Sanda, Hyogo 669-1330, Japan}
\email{ohsugi@kwansei.ac.jp}

\keywords{flow polytopes, order polytopes, $st$-planar graphs}
\subjclass[2020]{Primary 52B20; Secondary 05C10, 05B35, 06A07.}

\begin{document}

\begin{abstract}
Flow polytopes of directed acyclic graphs form a central class of lattice polytopes in algebraic, geometric, and enumerative combinatorics.
Order polytopes are one of the best understood families of lattice polytopes;
their Ehrhart theory, triangulations, volumes, and face structures are closely controlled by the combinatorics of the underlying posets.
M\'esz\'aros--Morales--Striker proved that the flow polytope of an $st$-planar directed acyclic graph is unimodularly equivalent to an order polytope.
In this paper, we prove a converse after contracting idle edges.
More precisely, for a directed acyclic graph $G$ with a unique source and a unique sink,
let $\widetilde G$ be the graph obtained from $G$ by successively contracting idle edges until none remain.
We prove that $\Fc(G)$ is unimodularly equivalent to an order polytope if and only if $\widetilde G$ is $st$-planar.
In addition, under a local three-good-neighbor condition, we prove that for a directed acyclic graph with a unique source, a unique sink, and no idle edges,
the graph is $st$-planar if and only if it avoids an explicit list of forbidden butterfly minors.
\end{abstract}

\maketitle

\section{Introduction}
Throughout this paper, we write a ``DAG" for a ``directed acyclic graph".
Flow polytopes encode nonnegative flows on directed graphs with prescribed netflow. They appear naturally in several parts of algebraic and enumerative combinatorics,
including the study of the Chan--Robbins--Yuen polytope~\cite{CRY}, Gelfand--Tsetlin polytopes~\cite{LMSD}, Kostant partition functions~\cite{BV}, root polytopes and faces of the alternating sign matrix polytope~\cite{MMS}.
Their normalized volumes, triangulations, lattice-point enumeration, and Ehrhart series often admit striking combinatorial interpretations.
Recent developments have also emphasized unimodular triangulations of flow polytopes~\cite{BraunCornejo}, volume inequalities for special classes of DAGs~\cite{BraunMcElroy}, and connections with Gorenstein polytopes and representation-theoretic structures~\cite{BraunCornejo, LMSD, vonBell}.

Order polytopes, introduced by Stanley \cite{Stanley}, are a foundational class of lattice polytopes attached to finite posets.
Their vertices are indexed by order filters, their normalized volumes are the numbers of linear extensions, and their canonical triangulations and Ehrhart series are well understood in terms of the underlying poset.
Therefore, when a flow polytope can be realized as an order polytope, one can transfer many questions about the former to the combinatorics of posets.
Recent work of Rietsch--Williams relates root polytopes, flow polytopes, and order polytopes to toric geometry, Hibi toric varieties, and mirror symmetry \cite{RW}.
M\'esz\'aros--Morales--Striker proved that flow polytopes of $st$-planar DAGs are unimodularly equivalent to order polytopes (\cite[Theorem 3.11]{MMS}).

The purpose of this paper is to prove a converse after contracting idle edges.
To avoid the trivial case, we assume that the flow polytope of a DAG $G$ is not $0$-dimensional throughout this paper.
Equivalently, $G$ is not a directed path.
We prove that a flow polytope is affinely equivalent to an order polytope only when the complete contraction of the graph is $st$-planar.
Here, an edge $uv$ is called \emph{idle} if it is the only outgoing edge of $u$, or the only incoming edge of $v$,
and the \emph{complete contraction} of $G$ is obtained by successively contracting idle edges until no idle edges remain.
Note that the complete contraction of $G$ is unique (Proposition~\ref{prop:idle}).
Since contracting idle edges preserves the flow polytope up to unimodular equivalence (Lemma~\ref{lem:new-butterfly}), complete contractions preserve flow polytopes up to unimodular equivalence.

We now state the first main theorem.

\begin{thm}\label{thm:main1}
Let $G$ be a DAG with a unique source $s$ and a unique sink $t$, and
let $\widetilde{G}$ be the complete contraction of $G$.
Then the following conditions are equivalent:
\begin{enumerate}[label={\rm (\roman*)}]
\item $\widetilde{G}$ is $st$-planar;
\item $\Fc(G)$ is unimodularly equivalent to an order polytope;
\item $\Fc(G)$ is affinely equivalent to an order polytope.
\end{enumerate}
\end{thm}

The implication (i) $\Rightarrow$ (ii) follows from the theorem appearing in M\'esz\'aros--Morales--Striker~\cite{MMS}.
Note that their terminology is slightly different from ours:
what they call a \emph{planar} DAG is precisely an $st$-planar DAG in the sense of this paper,
namely a DAG with a unique source $s$ and a unique sink $t$ such that $G\cup\{st\}$ is planar.
From \cite[Theorem 3.11]{MMS}, if $\widetilde{G}$ is $st$-planar, then $\Fc(\widetilde{G})$ is unimodularly equivalent to an order polytope.
Since $\Fc(G)$ and $\Fc(\widetilde{G})$ are unimodularly equivalent by Lemma \ref{lem:new-butterfly}, condition (ii) follows.
The implication (ii) \(\Rightarrow\) (iii) is immediate, since every unimodular equivalence is an affine equivalence.
Thus, the main content of this paper is the implication (iii) $\Rightarrow$ (i).
The key ingredient in the proof is a comparison between the facet-normal matroids of flow polytopes (Proposition~\ref{prop:flow-facet-matroid}) and those of order polytopes (Proposition~\ref{prop:order-facet-matroid}); see Section~\ref{sec:(iii)to(i)} for details.

\begin{example}
Let $\overrightarrow{K_n}$ denote the complete DAG on $\{1,2,\ldots,n\}$ whose edges are directed from smaller vertices to larger vertices.
M\'esz\'aros--Morales--Striker \cite[Remark~4.4]{MMS} asked whether $\mathcal F(\overrightarrow{K_7})$ is unimodularly equivalent to an order polytope.
Theorem~\ref{thm:main1} gives a negative answer to this question, not only for $\mathcal F(\overrightarrow{K_7})$ but also for $\mathcal F(\overrightarrow{K_n})$ for every $n \geq 7$.
In fact, for example, the complete contraction of $\overrightarrow{K_7}$ is $\Fc_5 \cup \{15\}$, where $\Fc_5$ appears in Section \ref{sec:onlyif}.
Since $\Fc_5 \cup \{15\}$ is not $st$-planar, Theorem~\ref{thm:main1} implies that $\mathcal F(\overrightarrow{K_7})$ is not affinely equivalent to any order polytope.
In particular, $\mathcal F(\overrightarrow{K_7})$ is not unimodularly equivalent to any order polytope.
\end{example}

Next, we study a forbidden butterfly minor criterion for $st$-planar graphs.
A butterfly minor is obtained by deleting directed edges and contracting idle edges.
The notions needed in the statement below are recalled in Section~\ref{sec:prelim}.
Let $G=(V,E)$ be a DAG with a unique source $s$ and a unique sink $t$.
We call a vertex in $V \setminus \{s,t\}$ \emph{internal}.
For an internal vertex $v$, let $N^+(v)$ and $N^-(v)$ be the out-neighbors and in-neighbors of $v$ inside $V \setminus\{s,t\}$, respectively.
\begin{itemize}
    \item A vertex $u\in N^+(v)$ is called {\it good} if there are at least two parallel edges $u\to t$ and an edge $s\to u$.
    \item A vertex $u\in N^-(v)$ is called {\it good} if there are at least two parallel edges $s\to u$ and an edge $u\to t$.
    \item We say that $G$ satisfies the \textit{three-good-neighbor condition} \ref{Hvertex} if

\medskip

\begin{enumerate}[label=(GN)]
\item\label{Hvertex}
every vertex $v\in V \setminus\{s,t\}$ with $|N^+(v)\cup N^-(v)|\ge3 $ has at least three good neighbors.
\end{enumerate}
\end{itemize}

\smallskip

Figure~\ref{fig:H2} shows the two representative distributions of three good neighbors. The other two distributions are obtained by reversing all arrows.

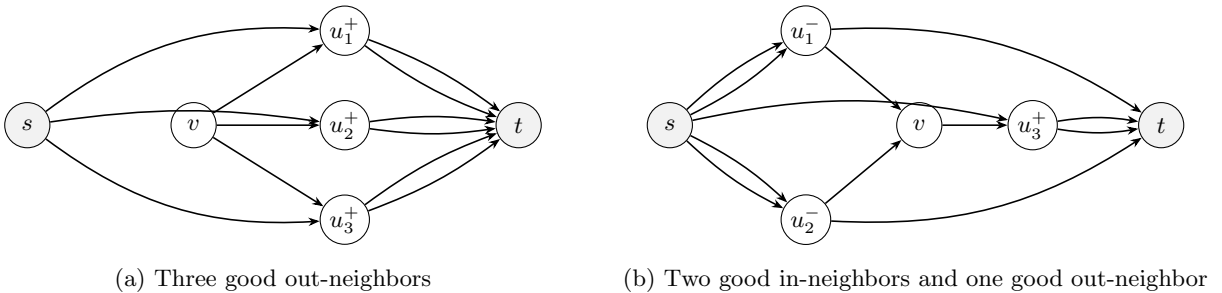
\begin{figure}[h]
\centering

\begin{tikzpicture}[ >=Latex, vertex/.style={ circle, draw, fill=white, inner sep=1.2pt, minimum size=17pt }, boundary/.style={ circle, draw, fill=gray!10, inner sep=1.2pt, minimum size=17pt }, edge/.style={
-{Stealth[length=1.5mm]},semithick}, every node/.style={ font=\small } ]
\node[boundary] (sA) at (0,0) {$s$};
\node[vertex] (vA) at (2.2,0) {$v$};
\node[boundary] (tA) at (6.5,0) {$t$};
\node[vertex] (a1) at (4.2,1.25) {$u_1^+$};
\node[vertex] (a2) at (4.2,0) {$u_2^+$};
\node[vertex] (a3) at (4.2,-1.25) {$u_3^+$};
\draw[edge] (vA) -- (a1);
\draw[edge] (vA) -- (a2);
\draw[edge] (vA) -- (a3);
\draw[edge] (sA) to[bend left=20] (a1);
\draw[edge] (sA) to[bend left=8] (a2);
\draw[edge] (sA) to[bend right=20] (a3);
\draw[edge] (a1) to[bend left=8] (tA);
\draw[edge] (a1) to[bend right=8] (tA);
\draw[edge] (a2) to[bend left=8] (tA);
\draw[edge] (a2) to[bend right=8] (tA);
\draw[edge] (a3) to[bend left=8] (tA);
\draw[edge] (a3) to[bend right=8] (tA);
\node at (3.25,-2.05) {(a) Three good out-neighbors};
\begin{scope}[xshift=8.5cm]
\node[boundary] (sB) at (0,0) {$s$};
\node[vertex] (vB) at (3.3,0) {$v$};
\node[boundary] (tB) at (6.5,0) {$t$};
\node[vertex] (b1) at (1.8,1.25) {$u_1^-$};
\node[vertex] (b2) at (1.8,-1.25) {$u_2^-$};
\node[vertex] (c1) at (4.8,0) {$u_3^+$};
\draw[edge] (sB) to[bend left=8] (b1);
\draw[edge] (sB) to[bend right=8] (b1);
\draw[edge] (sB) to[bend left=8] (b2);
\draw[edge] (sB) to[bend right=8] (b2);
\draw[edge] (b1) -- (vB);
\draw[edge] (b2) -- (vB);
\draw[edge] (b1) to[bend left=18] (tB);
\draw[edge] (b2) to[bend right=18] (tB);
\draw[edge] (sB) to[bend left=12] (c1);
\draw[edge] (vB) -- (c1);
\draw[edge] (c1) to[bend left=9] (tB);
\draw[edge] (c1) to[bend right=9] (tB);
\node at (3.25,-2.05) {(b) Two good in-neighbors and one good out-neighbor};
\end{scope}
\end{tikzpicture} \caption{Representative configurations in condition \ref{Hvertex}.
The configurations obtained by reversing all arrows are omitted.} \label{fig:H2}

\end{figure}

We now state the second main theorem.

\begin{thm}\label{thm:main2}
Let $G$ be a DAG with a unique source $s$ and a unique sink $t$, and assume that $G$ has no idle edges.
Suppose that $G$ satisfies  the three-good-neighbor condition \ref{Hvertex}.  Then $G$ is $st$-planar
if and only if $G$ contains neither a DAG in
\[\{\Fc_5,\Fc_{33}^{(1a)},\Fc_{33}^{(1b)},\Fc_{33}^{(1c)},\Fc_{33}^{(3a)},\Fc_{33}^{(4)},\Fc_{33}^{(5a)},\Fc_{33}^{(5b)}\} \cup \{\Dc_n\}_{n \geq 2}
\]
nor the reversal of such a DAG as a butterfly minor.
\end{thm}

For the forbidden graphs listed in Theorem \ref{thm:main2}, see Section~\ref{sec:onlyif}.
Since butterfly minors only allow contractions of idle edges,
they are substantially more restrictive than ordinary graph minors.
Consequently, the Robertson--Seymour theory does not apply,
and there is no general reason to expect a finite excluded-minor
characterization for butterfly-minor-closed classes. See Remark~\ref{rmk:forbidden}.

\medskip
\noindent\textbf{Organization.}
Section~\ref{sec:prelim} recalls flow polytopes, order polytopes, butterfly minors, and $st$-planarity.
Section~\ref{sec:(iii)to(i)} proves Theorem \ref{thm:main1} (iii) $\Rightarrow$ (i).
In Section~\ref{sec:onlyif}, the forbidden graphs are displayed, and we see that the only-if direction of Theorem~\ref{thm:main2} directly follows.
Section~\ref{sec:c-to-b} proves the if direction of Theorem~\ref{thm:main2} by a graph-theoretic analysis of Kuratowski subdivisions.

\medskip
\noindent\textbf{Acknowledgments.}
The first named author A. H. is supported by JSPS KAKENHI JP24K00521, and the second named author H. O. is supported by JSPS KAKENHI JP24K00534.

\bigskip
\section{Preliminaries}\label{sec:prelim}

In this section, we describe necessary notions used in this paper.
\subsection{Flow polytopes}
Throughout this paper, DAGs may have parallel edges but no loops.
Let $G=(V,E)$ be a finite DAG with a unique source $s$ and a unique sink $t$.
The \emph{flow polytope} of $G$ with unit flow is
\[\Fc(G)=\left\{(x_e)_{e \in E} \in \RR_{\geq 0}^E: \sum_{e\in \operatorname{out}(v)}x_e-\sum_{e\in \operatorname{in}(v)}x_e=
\begin{cases}
1,&v=s,\\
-1,&v=t,\\
0,&v\in V \setminus \{s,t\}
\end{cases}\right\},\]
where $\mathrm{out}(v)$ (resp. $\mathrm{in}(v)$) denotes the set of arrows in $G$ whose tail (resp. head) is $v$.
By definition, we see that the reversal of all edges gives the same flow polytope.
Its vertices are the characteristic vectors of directed \(s\)-to-\(t\) paths. Hence, \(\mathcal F(G)\) is a \(0/1\)-polytope, and in particular a lattice polytope.
When $G$ is connected, its dimension is $\dim \Fc(G)=|E|-|V|+1$.

\subsection{Order polytopes}
For a finite poset $P$ equipped with a partial order $\prec$, the \emph{order polytope} of $P$ is
\[\Oc(P)=\{(x_p)_{p\in P}\in \RR^P: 0\leq x_p\leq 1 \text{ for all }p\in P,\ x_p\leq x_q \text{ whenever }p\prec q\}. \]
We collect some fundamental facts on $\Oc(P)$. See \cite{Stanley} for the details.
\begin{itemize}
    \item We have $\dim \Oc(P)=|P|$.
    \item The vertices of $\Oc(P)$ are characteristic vectors of order filters of $P$, and order filters are naturally indexed by their sets of minimal elements, equivalently by antichains of $P$.
    Hence, the number of vertices of $\Oc(P)$ is equal to the number of antichains of $P$.
    \item The facets of $\Oc(P)$ are indexed by the cover relations in $P\cup\{\hat0,\hat1\}$,
    where $\hat0$ (resp. $\hat1$) is a new minimal (resp. maximal) element outside of $P$.
    Thus, the number of facets of $\Oc(P)$ is equal to the number of edges in the Hasse diagram of $P\cup\{\hat0,\hat1\}$.
    Throughout this paper, we use the notation $p \lessdot q$ to indicate that $q$ covers $p$.
    \item For a lattice polytope $Q \subset \RR^d$, let
    \[\mathrm{codeg}(Q)=\min\{k \in \ZZ_{>0} : kQ^\circ \cap \ZZ^d \neq \emptyset\},\]
    where $Q^\circ$ denotes the relative interior of $Q$. We call $\mathrm{codeg}(Q)$ the \textit{codegree} of $Q$.
    For the case of order polytopes, by the description of $\Oc(P)$, we see that
    \[\codeg(\Oc(P))=\ell(P)+2,\] where $\ell(P)$ is the maximum number of cover relations in a chain of $P$.
    \item Given two finite posets $P$ and $Q$, their {\it ordinal sum} $P\oplu Q$ is the poset on the disjoint union $P\sqcup Q$ whose partial order extends the orders of $P$ and $Q$ and satisfies $p \prec q$ for all $p\in P$ and $q\in Q$.
\end{itemize}
We see that the order polytopes of $P$ and its reversal are unimodularly equivalent.
\begin{lem}\label{lem:faces-order}
Every face of an order polytope is unimodularly equivalent to an order polytope.
\end{lem}
\begin{proof}
A face of an order polytope is obtained by setting some inequalities $x_p\geq0$, $x_p\leq1$, or $x_p\leq x_q$ to equalities.  These equalities identify some elements of $P$, or force them to be equal to $0$ or $1$.  After deleting forced coordinates and quotienting by the identifications, one obtains the order polytope of another poset.
\end{proof}

\subsection{Unimodular equivalence}

For a nonempty polytope $P\subset \RR^n$, let \[L_P:=\operatorname{span}_{\RR}\{x-y:x,y\in P\}\] and \[\aff(P):=x_0+L_P,\]
where $x_0\in P$. In particular, $L_P$ is the direction space of $\aff(P)$.

Let $P\subset \RR^m$ and $Q\subset \RR^n$ be lattice polytopes.
\begin{itemize}
    \item We say that $P$ and $Q$ are \emph{affinely equivalent} if there exists an affine bijection
    \[\varphi:\aff(P)\longrightarrow \aff(Q)\] such that $\varphi(P)=Q$.
    \item We say that $P$ and $Q$ are \emph{unimodularly equivalent} if there exists an affine bijection $\varphi:\aff(P)\longrightarrow \aff(Q)$ such that $\varphi(P)=Q$ and $\varphi$ restricts to a bijection \[\aff(P)\cap \ZZ^m \longrightarrow \aff(Q)\cap \ZZ^n.\]
    Equivalently, after choosing lattice bases of the affine lattices $\aff(P)\cap \ZZ^m$ and $\aff(Q)\cap \ZZ^n$, the map $\varphi$ is of the form
\[x\longmapsto Ux+b,\]
where $U\in \operatorname{GL}_d(\ZZ)$ and $b\in \ZZ^d$ with $d=\dim P=\dim Q$.
Unimodular equivalence preserves many properties.
\end{itemize}
Note that what the authors of \cite{MMS} call integrally equivalent is the same as what we call unimodularly equivalent in this paper.

\subsection{Butterfly minors and $st$-planarity}
Recall that an edge $e=uv$ of a DAG is called \emph{idle} if $e$ is the only outgoing edge of $u$ or if $e$ is the only incoming edge of $v$.
A \emph{butterfly contraction} of a DAG $G$ is the contraction of an idle edge.
A \emph{butterfly minor} is obtained by repeatedly deleting edges and contracting idle edges.
A \emph{complete contraction} of $G$ is obtained by successively contracting idle edges until no idle edges remain.
The complete contraction is stated to be unique in \cite[Definition 2.4]{BBDH}. Since this uniqueness plays an important role in what follows, we include a proof for completeness.

\begin{prop}\label{prop:idle}
Let $G$ be a DAG. Then any two complete contractions of $G$ are isomorphic.
\end{prop}
\begin{proof}
Since $G$ is a DAG, contracting an idle edge again gives a DAG.
Indeed, let $e=u\to v$ be idle.
If $e$ is the only outgoing edge of $u$, then every edge leaving the contracted vertex in $G/e$ comes from an edge leaving $v$ in $G$.
Hence any directed cycle in $G/e$ passing through the contracted vertex lifts to a directed cycle in $G$.
The case where $e$ is the only incoming edge of $v$ is dual.
Thus contracting an idle edge stays within the class of DAGs.

We apply Newman's lemma to the reduction relation on the isomorphism classes of finite directed multigraphs defined by idle-edge contractions.
Each contraction decreases the number of vertices by one, so every sequence of idle-edge contractions terminates.
Thus it remains to prove local confluence.
Let $e$ and $f$ be two distinct idle edges of $G$.
We show that the two one-step contractions $G/e$ and $G/f$ have a common descendant.

\bigskip

\noindent
{\bf Case 1.}
($e$ and $f$ have no common endpoint.)
Then contracting one of them does not affect the idleness of the other.
Thus the image of $f$ is idle in $G/e$, the image of $e$ is idle in $G/f$, and contracting both edges gives the same directed multigraph, up to the evident isomorphism.

\bigskip

\noindent
{\bf Case 2.}
($e$ and $f$ have a common endpoint.)
Since $e$ and $f$ are idle, they are not parallel.

\bigskip

\noindent
{\bf Case 2.1.}
($e$ and $f$ have the same tail.)
Write
\[
e=u\to v,\qquad f=u\to w,
\qquad v\neq w.
\]
Since $e$ and $f$ are idle, $e$ is the only incoming edge of $v$, and $f$ is the only incoming edge of $w$.
After contracting $e$, the image of $f$ is still the only incoming edge of $w$, and hence is idle.
Similarly, after contracting $f$, the image of $e$ is still idle.
Contracting both edges gives the same graph.
The case where $e$ and $f$ have the same head is dual.

\bigskip

\noindent
{\bf Case 2.2.}
($e$ and $f$ form a directed path.)
Let
\[
e=u\to v,\qquad f=v\to w.
\]
If $e$ is the only outgoing edge of $u$, then the image of $e$ remains idle after contracting $f$.
Moreover, after contracting $e$, the image of $f$ is idle:
\begin{itemize}
    \item
    If $f$ is the only incoming edge of $w$,
    then so is the image of $f$.
    \item
    If $f$ is the only outgoing edge of $v$,
    then the image of $f$ is the only outgoing edge of the contracted vertex.
\end{itemize}
Thus the two contractions can be performed in either order and lead to the same graph.
The dual argument applies if $f$ is idle because it is the only incoming edge of $w$.

The only remaining case is the following:
\[e=u\to v \text{ is idle because it is the only incoming edge of }v,\]
and
\[f=v\to w \text{ is idle because it is the only outgoing edge of }v.\]
In this case, $v$ has precisely one incoming edge and precisely one outgoing edge, namely $e$ and $f$.
Hence $v$ is incident with no other edge.
Let $H$ be the directed multigraph obtained from $G$ by deleting $v$ and the two edges $e,f$, and by adding a single edge $u\to w$.
Then both $G/e$ and $G/f$ are naturally isomorphic to $H$: in $G/e$, the image of $f$ becomes the edge $u\to w$, while in $G/f$, the image of $e$ becomes the edge $u\to w$.
Thus the two one-step contractions $G/e$ and $G/f$ are already isomorphic.

\bigskip

This proves local confluence.
Since the reduction relation is terminating and locally confluent, it is confluent by Newman's lemma.
Consequently, all maximal sequences of idle-edge contractions starting from $G$ end in isomorphic directed multigraphs with no idle edges.
Thus the complete contraction of $G$ is unique up to isomorphism.
\end{proof}

\begin{rmk}
     The description in \cite[Definition~2.4]{BBDH} should be interpreted with some care: contracting an idle edge does not necessarily decrease the number of idle edges by exactly one, since the contraction may destroy the idleness of another edge. The complete contraction is therefore more naturally defined as the result of repeatedly contracting idle edges until none remain. By Proposition~\ref{prop:idle}, this result is unique up to isomorphism.
 \end{rmk}

We shall use the following elementary observation together with the standard facet description of flow polytopes (see, e.g., \cite{BBDH, DKK}).
We include a proof for completeness, since this description plays a central role in this paper.
\begin{lem}\label{lem:new-butterfly}
Let $\widetilde{G}$ be the complete contraction of $G$.
Then $\Fc(G)$ and $\Fc(\widetilde{G})$ are unimodularly equivalent, and the facets of $\Fc(G)$ are naturally indexed by the edges of $\widetilde{G}$.
In particular, if $G$ has no idle edges, then every coordinate inequality $x_e\geq 0$  $(e\in E)$ defines a distinct facet of $\Fc(G)$.
\end{lem}
\begin{proof}
    Let \(e=uv\) be an idle edge. If \(e\) is the only outgoing edge of \(u\), then $x_e=\sum_{f\in\operatorname{in}(u)}x_f$ when \(u\neq s\), while \(x_e=1\) when \(u=s\). Similarly, if \(e\) is the only incoming edge of \(v\), then $ x_e=\sum_{g\in\operatorname{out}(v)}x_g $ when \(v\neq t\), while \(x_e=1\) when \(v=t\). Thus, deleting the coordinate \(x_e\) and restoring it by the appropriate formula above give mutually inverse affine maps preserving the corresponding affine lattices. Hence, $\Fc(G)$ is unimodularly equivalent to $\Fc(G/e)$.
    Iterating this argument, it follows that $\Fc(G)$ is unimodularly equivalent to $\mathcal F(\widetilde G)$.
    By construction, \(\widetilde G\) has no idle edges.
    From \cite[Proposition~2]{DKK}, the facets of the nonnegative flow cone are indexed by its non-idle edges.
    Taking the unit flow section shows that the facets of \(\mathcal F(\widetilde G)\) are precisely \[ \mathcal F(\widetilde G)\cap\{x_e=0\}, \qquad e\in E(\widetilde G), \]
    as required.
\end{proof}

\begin{lem}\label{lem:butterfly-deletion}
Let $H$ be a graph obtained by deleting an edge of $G$.
If $\Fc(H)$ is nonempty, then $\Fc(H)$ is unimodularly equivalent to a face of $\Fc(G)$.
\end{lem}
\begin{proof}
Deleting an edge $e$ of $G$ corresponds to taking the face $x_e=0$ of $\Fc(G)$.
\end{proof}

Lemmas~\ref{lem:new-butterfly} and \ref{lem:butterfly-deletion} immediately prove the following.

\begin{cor}\label{cor:butterfly-face}
Let $H$ be a butterfly minor of $G$.
If $\Fc(H)$ is nonempty, then $\Fc(H)$ is unimodularly equivalent to a face of $\Fc(G)$.
\end{cor}

In \cite{MMS}, a connected DAG on the vertex set $\{1,2,\ldots,n\}$, whose edges are directed from smaller labels to larger labels, is called \emph{planar}
if it admits a planar drawing such that, whenever vertex $i$ is placed at $(x_i,y_i)$, one has $x_i<x_j$ for $i<j$.
In graph-theoretic terminology, after rotating the picture by $90$ degrees, this is an upward planar drawing.
We will use the following standard characterization of upward planar digraphs.
\begin{prop}[\cite{DBT,Kel}]\label{prop:st}
Let $G$ be a DAG. Then $G$ has an upward planar drawing if and only if $G$ is a spanning subgraph of a planar $st$-graph.
In particular, if $G$ has a unique source $s$ and a unique sink $t$, then $G$ has an upward planar drawing if and only if $G$ is $st$-planar, i.e., $G\cup\{st\}$ is planar.
\end{prop}
Thus, in the setting of this paper, the three notions, $st$-planarity, planarity in the sense of \cite{MMS} and
existence of an upward planar drawing, coincide.
Throughout this paper, we use the term $st$-planar.

\bigskip
\section{Proof of {\rm (iii)} $\Rightarrow$ {\rm (i)}}\label{sec:(iii)to(i)}

In this section, we prove that if $G$ has no idle edges and is not $st$-planar, then $\Fc(G)$ is not affinely equivalent to any order polytope.
For background on matroid theory, we refer to Oxley~\cite{Oxley}.
For vector configurations, oriented matroids, and Gale duality in the setting of polytopes, see Ziegler~\cite[Chapter~6]{Ziegler}.

For a real matrix $A$, we denote by $M(A)$ its linear matroid associated with $A$: a subset of columns is independent if and only if it is linearly independent over $\RR$.

For a graph \(G\), we denote by \(M(G)\) the graphic matroid of its underlying undirected multigraph.
The dual matroid \(M(G)^*\) is called the cographic matroid of \(G\).

For each facet \(F\) of $P$, choose a nonzero normal functional $\nu_F\in L_P^*$ whose kernel is the direction space of \(\operatorname{aff}(F)\).
The \emph{facet-normal matroid} of $P$, denoted by \(M_{\mathrm{fac}}(P)\), is the linear matroid on the set of facets of $P$ represented by $\{\nu_F:F\text{ is a facet of }P\} \subset L_P^*$.
This matroid is independent of the choices of the nonzero scalar multiples of the normal functionals.

\begin{lem}\label{lem:facet-matroid-affine-invariant}
If polytopes \(P_1\) and \(P_2\) are affinely equivalent, then $M_{\mathrm{fac}}(P_1)\cong M_{\mathrm{fac}}(P_2)$. \end{lem}

\begin{proof} Let $\varphi:\operatorname{aff}(P_1)\rightarrow \operatorname{aff}(P_2)$,  $x\mapsto Ax+b$ be an affine equivalence.
Then $A$ induces a linear isomorphism $A:L_{P_1}\to L_{P_2}$ and
the map \(\varphi\) induces a bijection between the facets of \(P_1\) and those of \(P_2\).
If \(\nu_F\in L_{P_1}^*\) is a normal functional of a facet \(F\), then a normal functional of \(\varphi(F)\) is $\nu_F\circ A^{-1}\in L_{P_2}^*$.
Since the dual map
$(A^{-1})^*:L_{P_1}^*\rightarrow L_{P_2}^*$ is invertible, it preserves all linear dependence relations among the normal functionals.
Therefore, the corresponding facet-normal matroids are isomorphic. \end{proof}

For a real matrix \(A\), let \(\operatorname{row}(A)\) denote the linear subspace spanned by its rows. For a linear subspace \(W\subset\mathbb R^E\), let \[ W^\perp := \{x\in\mathbb R^E:\langle x,w\rangle=0 \text{ for every }w\in W\}, \]
where $\langle \cdot, \cdot \rangle$ denotes the usual inner product of $\RR^E$.

\begin{lem}\label{lem:orthogonal-representations} Let \(A\) and \(C\) be real matrices with the same number of columns.
Suppose that $\operatorname{row}(C) = \operatorname{row}(A)^\perp$. Then $M(C)\cong M(A)^*$.
\end{lem}
\begin{proof}
This is the standard orthogonal-complement representation of the dual of a linear matroid. See, for example, \cite[Chapter~2]{Oxley}.
\end{proof}

\begin{prop}\label{prop:flow-facet-matroid}
Let $G=(V,E)$ be a DAG with a unique source $s$ and a unique sink $t$.
Assume that $G$ has no idle edges. Then we have
$M_{\mathrm{fac}}(\mathcal F(G)) \cong M(G)^*$.
\end{prop}
\begin{proof} Let \(B_G\) be the signed vertex-edge incidence matrix of \(G\), and let \(b\in\mathbb R^V\) be the unit-netflow vector,
with entries \(1\) at \(s\), \(-1\) at \(t\), and \(0\) at every vertex in $V \setminus \{s,t\}$. Then \[ \mathcal F(G) = \{x\in\mathbb R_{\geq0}^E:B_Gx=b\}. \]
Since $G$ has a unique source $s$ and a unique sink $t$, every vertex is reachable from $s$ and reaches $t$.
Thus, every edge of $G$ is contained in a directed $s$-to-$t$ path.
Fix such a path $P_e$ containing $e$ for each edge $e$ and let $x^\circ=\frac{1}{|E|}\sum_{e\in E}\chi_{P_e}$, where $\chi_P \in \RR^E$ denotes the characteristic vector of a path $P$ in $G$.
Then $x^\circ$ belongs to $\Fc(G)$ and all of its coordinates are positive.
Since $\Fc(G)=A_G\cap \RR_{\geq0}^E$, where
\[A_G:=\{x\in\RR^E:B_Gx=b\},\]
the point $x^\circ$ lies in the relative interior of $\Fc(G)$ with respect to $A_G$.
Therefore, $\Fc(G)$ has nonempty relative interior in $A_G$, and hence
\[\aff(\Fc(G))=A_G,\]
and the direction space of \(\operatorname{aff}(\mathcal F(G))\) is
\[ L:=L_{\mathcal F(G)}=\ker B_G. \]
By Lemma~\ref{lem:new-butterfly}, the facets of \(\mathcal F(G)\) are precisely \[ \mathcal F(G)\cap\{x_e=0\}, \qquad e\in E. \]
Let $\varepsilon_e:\mathbb R^E\longrightarrow\mathbb R$,  $\varepsilon_e(x)=x_e$, be the \(e\)-th coordinate functional.
Relative to the affine hull of \(\mathcal F(G)\), a normal functional of the facet \(\mathcal F(G)\cap\{x_e=0\}\) is the restriction $\varepsilon_e|_L\in L^*$.
Thus, \(M_{\mathrm{fac}}(\mathcal F(G))\) is represented by \[ \{\varepsilon_e|_L:e\in E\}. \]
Choose a matrix \(C\) whose rows form a basis of \(L=\ker B_G\).
With respect to the basis of $L$ given by the rows of $C$,
the functional $\varepsilon_e|_L$
is represented by the $e$-th column of $C$.
Hence \[ M_{\mathrm{fac}}(\mathcal F(G))\cong M(C). \]
Moreover, \[ \operatorname{row}(C) = \ker B_G = \operatorname{row}(B_G)^\perp. \]
By Lemma \ref{lem:orthogonal-representations}, \[ M(C)\cong M(B_G)^*. \] Finally, the linear matroid associated with the signed incidence matrix \(B_G\) is the graphic matroid of the underlying undirected multigraph of \(G\).
Therefore, \[ M(B_G)\cong M(G), \] and consequently \[ M_{\mathrm{fac}}(\mathcal F(G)) \cong M(G)^*,\]
as desired.
\end{proof}

Let  \(P=\{p_1,\dots,p_d\}\) be a nonempty poset
and let $\widehat P:=P\sqcup\{\widehat 0,\widehat 1\}$.
Let \(K(P)\) be the graph obtained from the Hasse diagram of \(\widehat P\) by identifying \(\widehat 0\) and \(\widehat 1\).
Here, we regard $K(P)$ as a multigraph; parallel edges may occur after identifying $\widehat 0$ and $\widehat 1$.
The following observation is an immediate consequence of Stanley's facet description of order polytopes \cite[Section~1]{Stanley}.

\begin{prop}\label{prop:order-facet-matroid}
Let \(P=\{p_1,\dots,p_d\}\) be a nonempty poset. Then $M_{\mathrm{fac}}(\mathcal O(P)) \cong M(K(P))$.
\end{prop}
\begin{proof}
It is known \cite{Stanley} that
the order polytope \(\mathcal O(P) \subset \mathbb R^d\) is a $d$-dimensional polytope whose facets are indexed by the cover relations of \(\widehat P\).
More precisely, the cover relations \[ \widehat 0\lessdot p,\qquad p\lessdot q,\qquad p\lessdot\widehat 1 \] give the facet inequalities \[ x_p\geq0,\qquad x_q-x_p\geq0,\qquad 1-x_p\geq0, \] respectively. Up to multiplication by a nonzero scalar, their normal functionals are \[ \varepsilon_p,\qquad \varepsilon_q-\varepsilon_p,\qquad -\varepsilon_p, \] where
$\varepsilon_p:\mathbb R^d\longrightarrow\mathbb R$, $\varepsilon_p(x)=x_p$.
Let $v_0$ denote the vertex of $K(P)$ obtained by identifying $\widehat 0$ and $\widehat 1$.
Orient the edges of \(K(P)\) according to the cover relations of \(\widehat P\), and let \(D_{K(P)}\) be its signed incidence matrix. Delete the row corresponding to \(v_0\), obtaining a reduced signed incidence matrix $\overline D_{K(P)}$.
The columns of \(\overline D_{K(P)}\) corresponding to cover relations \(\widehat 0\lessdot p\), \(p\lessdot q\), \(p\lessdot\widehat 1\) are, up to sign,
\[\varepsilon_p, \qquad \varepsilon_q-\varepsilon_p, \qquad -\varepsilon_p,\] respectively.
Thus, the columns of \(\overline D_{K(P)}\) represent exactly the facet-normal functionals of \(\mathcal O(P)\), up to nonzero scalar multiples. Therefore, \[ M_{\mathrm{fac}}(\mathcal O(P)) \cong M(\overline D_{K(P)}) \cong M(K(P)), \] as required. \end{proof}

Now we are ready to give a proof of (iii) $\Rightarrow$ (i) in Theorem \ref{thm:main1}.
\begin{proof}[Proof of (iii) $\Rightarrow$ (i) in Theorem \ref{thm:main1}]
Complete contractions preserve flow polytopes up to unimodular equivalence by Lemma~\ref{lem:new-butterfly},
so we see that $\Fc(G)$ is unimodularly equivalent to $\Fc(\widetilde G)$.
Hence, it is enough to prove the assertion when the graph has no idle edges.

We prove the assertion by contraposition. Assume that $G$ is not $st$-planar. Let \(G'\) be obtained from \(G\) by adding a new edge \(e_0=st\).
Then $G'$ is not planar.
Since we assume that $G$ has no idle edges, neither does \(G'\).
Suppose that $\Fc(G)$ is affinely equivalent to the order polytope $\Oc(P)$ of a poset $P$. Then we have the following.
\begin{itemize}
    \item
    If $x \in \Fc(G')$ and $x_{e_0} =1$, then $x_e =0$ for all edges $e \ne e_0$.
    Hence \[ \Fc(G') = {\rm conv} (\{(z,0) : z \in \Fc(G)\} \cup \{({\bf 0},1)\}); \]
    \item
    If \(q\) is a new maximum element and \(P':=P\oplus\{q\}\), then
    \[\mathcal O(P')= {\rm conv} (\{(z,1) : z \in \Oc(P)\} \cup \{{\bf 0}\}).\]
\end{itemize}
Thus, $\Fc(G')$ is affinely equivalent to $\mathcal O(P'). $
By Lemma~\ref{lem:facet-matroid-affine-invariant} and Propositions~\ref{prop:flow-facet-matroid} and \ref{prop:order-facet-matroid},
\[ M(G')^* \cong M_{\mathrm{fac}}(\mathcal F(G')) \cong M_{\mathrm{fac}}(\mathcal O(P')) \cong M(K(P')). \]
Thus, the cographic matroid \(M(G')^*\) is graphic.
By Whitney's planarity criterion, a graph is planar if and only if the dual of its cycle matroid is graphic; see \cite[Theorem 5.2.2]{Oxley}.
This is a contradiction since $G'$ is not planar.
Therefore, $\Fc (G)$ is not affinely equivalent to any order polytope.
\end{proof}

\bigskip
\section{Forbidden butterfly minors of $st$-planar DAGs}\label{sec:onlyif}
In this section, we provide a list of forbidden butterfly minors used in Theorem~\ref{thm:main2} and several additional ones.
We see that all of them are non-$st$-planar.
Since $st$-planarity is closed under butterfly minors, no $st$-planar DAG can contain any of the displayed non-$st$-planar DAGs as a butterfly minor.
In particular, the only-if direction of Theorem~\ref{thm:main2} follows.

There are three types of DAGs: $K_5$-type, $K_{3,3}$-type, or type of even-cycle extension.
The underlying undirected graph of the DAG $\Fc_5$ becomes $K_5$ after adding the edge $15$ and deleting one copy of each of the parallel edges $12$ and $45$. For each $i\in\{1,2,3,4,5\}$, the underlying undirected graph of $\Fc_{33}^{(i)}$ contains $K_{3,3}-e$. The graphs corresponding to the branches of each type are minimal extensions without idle edges. For example, $\Fc_{33}^{(1 a)}, \Fc_{33}^{(1 b)}$, and $\Fc_{33}^{(1 c)}$ are minimal extensions of $\Fc_{33}^{(1)}$ having no idle edges.

\bigskip

\noindent
{\bf (The $K_5$-type forbidden minor)}: 
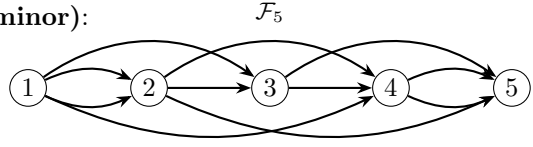
\vspace{-1cm}
\[
\begin{tikzpicture}[scale=0.8, every node/.style={fvertex}]
\node[flabel] at (4,1.25) {$\Fc_5$};

\node (1) at (0,0) {1};
\node (2) at (2,0) {2};
\node (3) at (4,0) {3};
\node (4) at (6,0) {4};
\node (5) at (8,0) {5};

\draw[fedge] (2) -- (3);
\draw[fedge] (3) -- (4);

\draw[fedge] (1) to[bend left=25] (2);
\draw[fedge] (1) to[bend right=25] (2);

\draw[fedge] (4) to[bend left=25] (5);
\draw[fedge] (4) to[bend right=25] (5);

\draw[fedge] (1) to[bend left=35] (3);
\draw[fedge] (2) to[bend left=35] (4);
\draw[fedge] (3) to[bend left=35] (5);

\draw[fedge] (1) to[bend right=25] (4);
\draw[fedge] (2) to[bend right=25] (5);
\end{tikzpicture}
\]

\noindent
{\bf (The $K_{3,3}$-type forbidden minors)}:
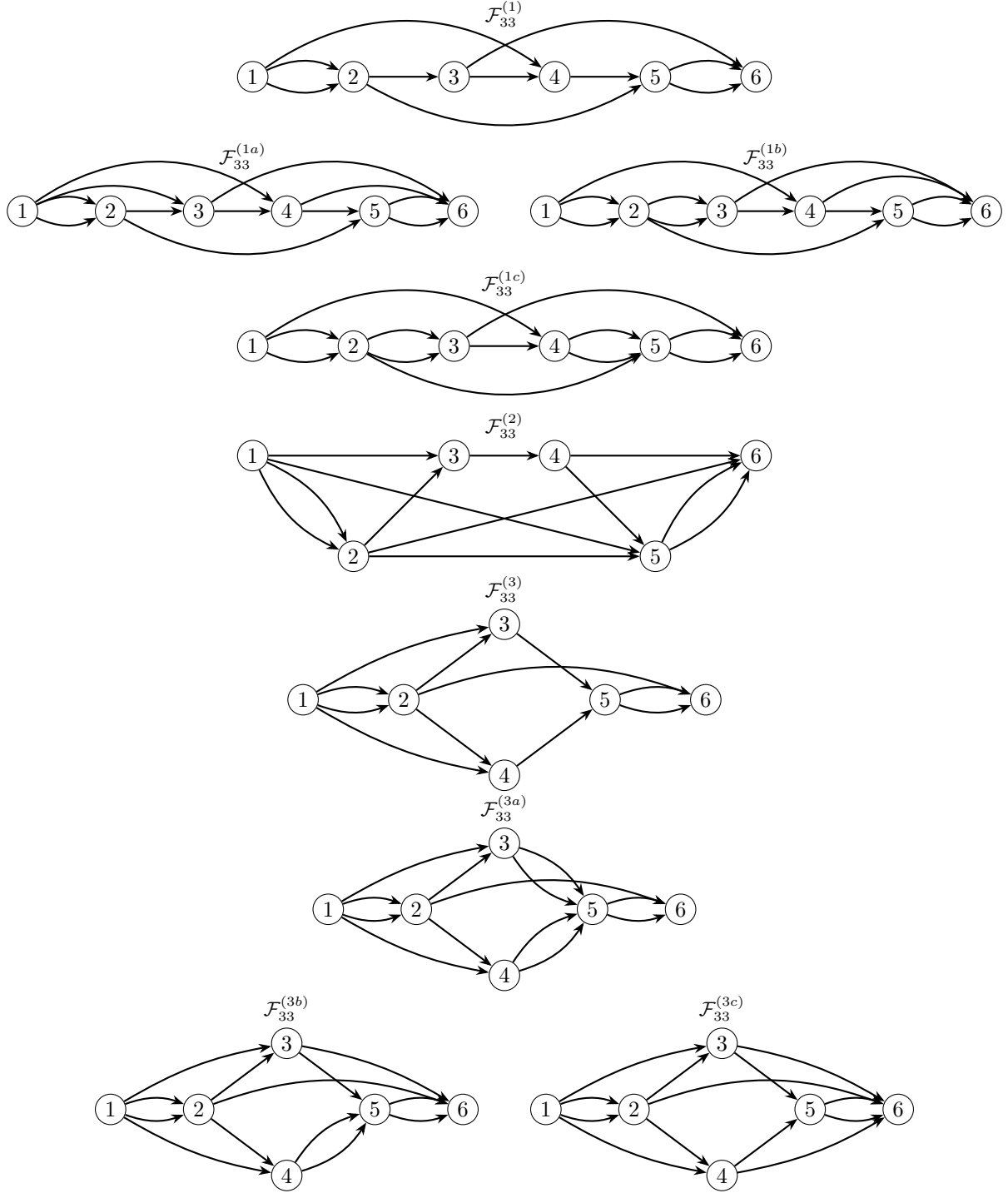
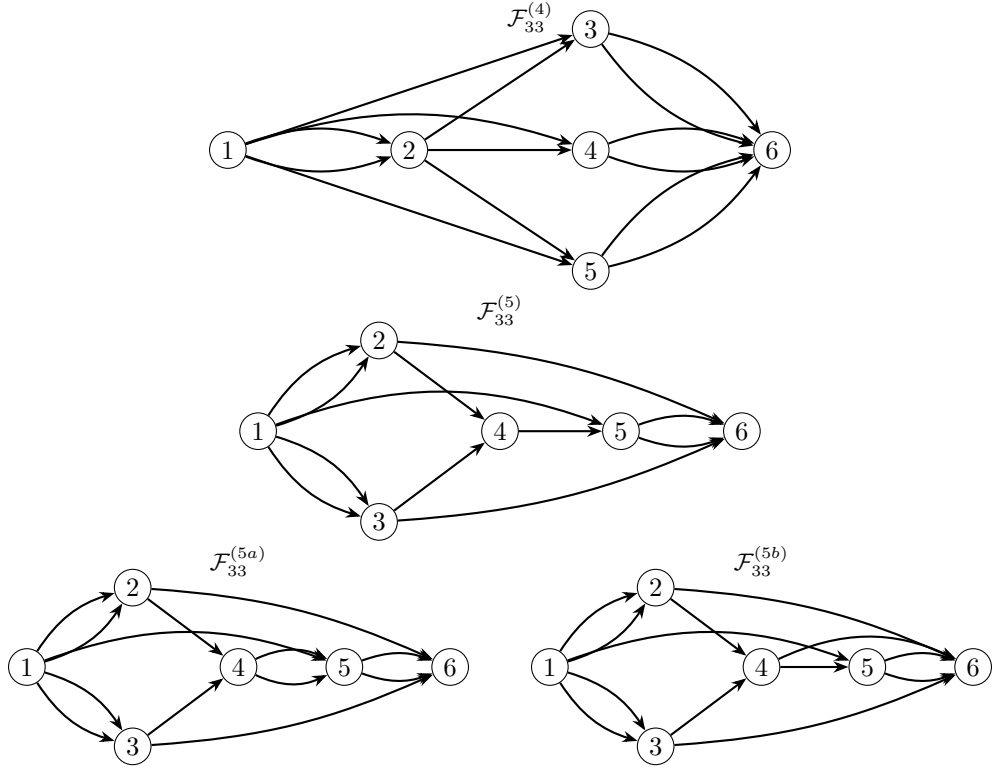
\begin{center}
\begin{tikzpicture}[scale=0.8, every node/.style={fvertex}]
\node[flabel] at (5,1.25) {$\Fc_{33}^{(1)}$};

\node (1) at (0,0) {1};
\node (2) at (2,0) {2};
\node (3) at (4,0) {3};
\node (4) at (6,0) {4};
\node (5) at (8,0) {5};
\node (6) at (10,0) {6};

\draw[fedge] (2) -- (3);
\draw[fedge] (3) -- (4);
\draw[fedge] (4) -- (5);

\draw[fedge] (1) to[bend left=25] (2);
\draw[fedge] (1) to[bend right=25] (2);

\draw[fedge] (5) to[bend left=25] (6);
\draw[fedge] (5) to[bend right=25] (6);

\draw[fedge] (1) to[bend left=35] (4);
\draw[fedge] (3) to[bend left=35] (6);

\draw[fedge] (2) to[bend right=30] (5);
\end{tikzpicture}

\begin{tikzpicture}[scale=0.7, every node/.style={fvertex}]
\node[flabel] at (5,1.25) {$\Fc_{33}^{(1a)}$};

\node (1) at (0,0) {1};
\node (2) at (2,0) {2};
\node (3) at (4,0) {3};
\node (4) at (6,0) {4};
\node (5) at (8,0) {5};
\node (6) at (10,0) {6};

\draw[fedge] (2) -- (3);
\draw[fedge] (3) -- (4);
\draw[fedge] (4) -- (5);

\draw[fedge] (1) to[bend left=25] (2);
\draw[fedge] (1) to[bend right=25] (2);

\draw[fedge] (5) to[bend left=25] (6);
\draw[fedge] (5) to[bend right=25] (6);

\draw[fedge] (1) to[bend left=25] (3);
\draw[fedge] (1) to[bend left=35] (4);
\draw[fedge] (3) to[bend left=35] (6);
\draw[fedge] (4) to[bend left=25] (6);

\draw[fedge] (2) to[bend right=30] (5);
\end{tikzpicture}
\qquad
\begin{tikzpicture}[scale=0.7, every node/.style={fvertex}]
\node[flabel] at (5,1.25) {$\Fc_{33}^{(1b)}$};

\node (1) at (0,0) {1};
\node (2) at (2,0) {2};
\node (3) at (4,0) {3};
\node (4) at (6,0) {4};
\node (5) at (8,0) {5};
\node (6) at (10,0) {6};

\draw[fedge] (3) -- (4);
\draw[fedge] (4) -- (5);

\draw[fedge] (1) to[bend left=25] (2);
\draw[fedge] (1) to[bend right=25] (2);

\draw[fedge] (5) to[bend left=25] (6);
\draw[fedge] (5) to[bend right=25] (6);

\draw[fedge] (2) to[bend left=25] (3);
\draw[fedge] (2) to[bend right=25] (3);

\draw[fedge] (1) to[bend left=35] (4);
\draw[fedge] (3) to[bend left=35] (6);
\draw[fedge] (4) to[bend left=35] (6);

\draw[fedge] (2) to[bend right=30] (5);
\end{tikzpicture}
\qquad
\begin{tikzpicture}[scale=0.8, every node/.style={fvertex}]
\node[flabel] at (5,1.25) {$\Fc_{33}^{(1c)}$};

\node (1) at (0,0) {1};
\node (2) at (2,0) {2};
\node (3) at (4,0) {3};
\node (4) at (6,0) {4};
\node (5) at (8,0) {5};
\node (6) at (10,0) {6};

\draw[fedge] (3) -- (4);

\draw[fedge] (1) to[bend left=25] (2);
\draw[fedge] (1) to[bend right=25] (2);

\draw[fedge] (5) to[bend left=25] (6);
\draw[fedge] (5) to[bend right=25] (6);

\draw[fedge] (2) to[bend left=25] (3);
\draw[fedge] (2) to[bend right=25] (3);

\draw[fedge] (4) to[bend left=25] (5);
\draw[fedge] (4) to[bend right=25] (5);

\draw[fedge] (1) to[bend left=35] (4);
\draw[fedge] (3) to[bend left=35] (6);

\draw[fedge] (2) to[bend right=30] (5);
\end{tikzpicture}

\begin{tikzpicture}[scale=0.8, every node/.style={fvertex}]
\node[flabel] at (5,0.6) {$\Fc_{33}^{(2)}$};

\node (1) at (0,0) {1};
\node (3) at (4,0) {3};
\node (4) at (6,0) {4};
\node (6) at (10,0) {6};

\node (2) at (2,-2) {2};
\node (5) at (8,-2) {5};

\draw[fedge] (1) -- (3);
\draw[fedge] (3) -- (4);
\draw[fedge] (4) -- (6);

\draw[fedge] (2) -- (3);
\draw[fedge] (4) -- (5);
\draw[fedge] (2) -- (5);
\draw[fedge] (1) -- (5);
\draw[fedge] (2) -- (6);

\draw[fedge] (1) to[bend left=20] (2);
\draw[fedge] (1) to[bend right=20] (2);
\draw[fedge] (5) to[bend left=20] (6);
\draw[fedge] (5) to[bend right=20] (6);
\end{tikzpicture}

\begin{tikzpicture}[scale=0.8, every node/.style={fvertex}]
\node[flabel] at (2,2.2) {$\Fc_{33}^{(3)}$};

\node (1) at (-2,0) {1};
\node (2) at (0,0) {2};
\node (3) at (2,1.5) {3};
\node (4) at (2,-1.5) {4};
\node (5) at (4,0) {5};
\node (6) at (6,0) {6};

\draw[fedge] (1) to[bend left=20] (2);
\draw[fedge] (1) to[bend right=20] (2);

\draw[fedge] (1) to[bend left=10] (3);
\draw[fedge] (1) to[bend right=10] (4);

\draw[fedge] (2) -- (3);
\draw[fedge] (2) -- (4);

\draw[fedge] (3) -- (5);

\draw[fedge] (4) -- (5);

\draw[fedge] (2) to[bend left=20] (6);

\draw[fedge] (5) to[bend left=20] (6);
\draw[fedge] (5) to[bend right=20] (6);
\end{tikzpicture}

\begin{tikzpicture}[scale=0.7, every node/.style={fvertex}]
\node[flabel] at (2,2.3) {$\Fc_{33}^{(3a)}$};

\node (1) at (-2,0) {1};
\node (2) at (0,0) {2};
\node (3) at (2,1.5) {3};
\node (4) at (2,-1.5) {4};
\node (5) at (4,0) {5};
\node (6) at (6,0) {6};

\draw[fedge] (1) to[bend left=20] (2);
\draw[fedge] (1) to[bend right=20] (2);

\draw[fedge] (1) to[bend left=10] (3);
\draw[fedge] (1) to[bend right=10] (4);

\draw[fedge] (2) -- (3);
\draw[fedge] (2) -- (4);

\draw[fedge] (3) to[bend right=20] (5);
\draw[fedge] (3) to[bend left=20] (5);

\draw[fedge] (4) to[bend right=20] (5);
\draw[fedge] (4) to[bend left=20] (5);

\draw[fedge] (2) to[bend left=20] (6);

\draw[fedge] (5) to[bend left=20] (6);
\draw[fedge] (5) to[bend right=20] (6);
\end{tikzpicture}

\begin{tikzpicture}[scale=0.7, every node/.style={fvertex}]
\node[flabel] at (2,2.3) {$\Fc_{33}^{(3b)}$};

\node (1) at (-2,0) {1};
\node (2) at (0,0) {2};
\node (3) at (2,1.5) {3};
\node (4) at (2,-1.5) {4};
\node (5) at (4,0) {5};
\node (6) at (6,0) {6};

\draw[fedge] (1) to[bend left=20] (2);
\draw[fedge] (1) to[bend right=20] (2);

\draw[fedge] (1) to[bend left=10] (3);
\draw[fedge] (1) to[bend right=10] (4);

\draw[fedge] (2) -- (3);
\draw[fedge] (2) -- (4);

\draw[fedge] (3) -- (5);
\draw[fedge] (3) to[bend left=10] (6);

\draw[fedge] (4) to[bend right=20] (5);
\draw[fedge] (4) to[bend left=20] (5);

\draw[fedge] (2) to[bend left=20] (6);

\draw[fedge] (5) to[bend left=20] (6);
\draw[fedge] (5) to[bend right=20] (6);
\end{tikzpicture}
\qquad
\begin{tikzpicture}[scale=0.7, every node/.style={fvertex}]
\node[flabel] at (2,2.3) {$\Fc_{33}^{(3c)}$};

\node (1) at (-2,0) {1};
\node (2) at (0,0) {2};
\node (3) at (2,1.5) {3};
\node (4) at (2,-1.5) {4};
\node (5) at (4,0) {5};
\node (6) at (6,0) {6};

\draw[fedge] (1) to[bend left=20] (2);
\draw[fedge] (1) to[bend right=20] (2);

\draw[fedge] (1) to[bend left=10] (3);
\draw[fedge] (1) to[bend right=10] (4);

\draw[fedge] (2) -- (3);
\draw[fedge] (2) -- (4);

\draw[fedge] (3) -- (5);
\draw[fedge] (3) to[bend left=10] (6);

\draw[fedge] (4) -- (5);
\draw[fedge] (4) to[bend right=10] (6);

\draw[fedge] (2) to[bend left=20] (6);

\draw[fedge] (5) to[bend left=20] (6);
\draw[fedge] (5) to[bend right=20] (6);
\end{tikzpicture}

\begin{tikzpicture}[scale=0.8, every node/.style={fvertex}]
\node[flabel] at (-1,2.2) {$\Fc_{33}^{(4)}$};

\node (1) at (-6,0) {1};
\node (2) at (-3,0) {2};
\node (3) at (0,2) {3};
\node (4) at (0,0) {4};
\node (5) at (0,-2) {5};
\node (6) at (3,0) {6};

\draw[fedge] (1) to[bend left=20] (2);
\draw[fedge] (1) to[bend right=20] (2);

\draw[fedge] (2) -- (3);
\draw[fedge] (2) -- (4);
\draw[fedge] (2) -- (5);

\draw[fedge] (1) -- (3);
\draw[fedge] (1) to[bend left=18] (4);
\draw[fedge] (1) -- (5);

\draw[fedge] (3) to[bend left=20] (6);
\draw[fedge] (3) to[bend right=20] (6);
\draw[fedge] (4) to[bend left=20] (6);
\draw[fedge] (4) to[bend right=20] (6);
\draw[fedge] (5) to[bend left=20] (6);
\draw[fedge] (5) to[bend right=20] (6);
\end{tikzpicture}

\begin{tikzpicture}[scale=0.8, every node/.style={fvertex}]
\node[flabel] at (4,2) {$\Fc_{33}^{(5)}$};

\node (1) at (0,0) {1};
\node (2) at (2,1.5) {2};
\node (3) at (2,-1.5) {3};
\node (4) at (4,0) {4};
\node (5) at (6,0) {5};
\node (6) at (8,0) {6};

\draw[fedge] (1) to[bend left=20] (2);
\draw[fedge] (1) to[bend right=20] (2);

\draw[fedge] (1) to[bend left=20] (3);
\draw[fedge] (1) to[bend right=20] (3);

\draw[fedge] (2) -- (4);
\draw[fedge] (3) -- (4);

\draw[fedge] (4) -- (5);

\draw[fedge] (5) to[bend left=20] (6);
\draw[fedge] (5) to[bend right=20] (6);

\draw[fedge] (1) to[bend left=20] (5);
\draw[fedge] (2) to[bend left=10] (6);
\draw[fedge] (3) to[bend right=10] (6);
\end{tikzpicture}

\begin{tikzpicture}[scale=0.7, every node/.style={fvertex}]
\node[flabel] at (4,2) {$\Fc_{33}^{(5a)}$};

\node (1) at (0,0) {1};
\node (2) at (2,1.5) {2};
\node (3) at (2,-1.5) {3};
\node (4) at (4,0) {4};
\node (5) at (6,0) {5};
\node (6) at (8,0) {6};

\draw[fedge] (1) to[bend left=20] (2);
\draw[fedge] (1) to[bend right=20] (2);

\draw[fedge] (1) to[bend left=20] (3);
\draw[fedge] (1) to[bend right=20] (3);

\draw[fedge] (2) -- (4);
\draw[fedge] (3) -- (4);

\draw[fedge] (4) to[bend left=25] (5);
\draw[fedge] (4) to[bend right=25] (5);

\draw[fedge] (5) to[bend left=20] (6);
\draw[fedge] (5) to[bend right=20] (6);

\draw[fedge] (1) to[bend left=20] (5);
\draw[fedge] (2) to[bend left=10] (6);
\draw[fedge] (3) to[bend right=10] (6);
\end{tikzpicture}
\qquad
\begin{tikzpicture}[scale=0.7, every node/.style={fvertex}]
\node[flabel] at (4,2) {$\Fc_{33}^{(5b)}$};

\node (1) at (0,0) {1};
\node (2) at (2,1.5) {2};
\node (3) at (2,-1.5) {3};
\node (4) at (4,0) {4};
\node (5) at (6,0) {5};
\node (6) at (8,0) {6};

\draw[fedge] (1) to[bend left=20] (2);
\draw[fedge] (1) to[bend right=20] (2);

\draw[fedge] (1) to[bend left=20] (3);
\draw[fedge] (1) to[bend right=20] (3);

\draw[fedge] (2) -- (4);
\draw[fedge] (3) -- (4);

\draw[fedge] (4) -- (5);
\draw[fedge] (4) to[bend left=25] (6);

\draw[fedge] (5) to[bend left=20] (6);
\draw[fedge] (5) to[bend right=20] (6);

\draw[fedge] (1) to[bend left=20] (5);
\draw[fedge] (2) to[bend left=10] (6);
\draw[fedge] (3) to[bend right=10] (6);
\end{tikzpicture}
\end{center}

\noindent
{\bf (The even-cycle family)}: 
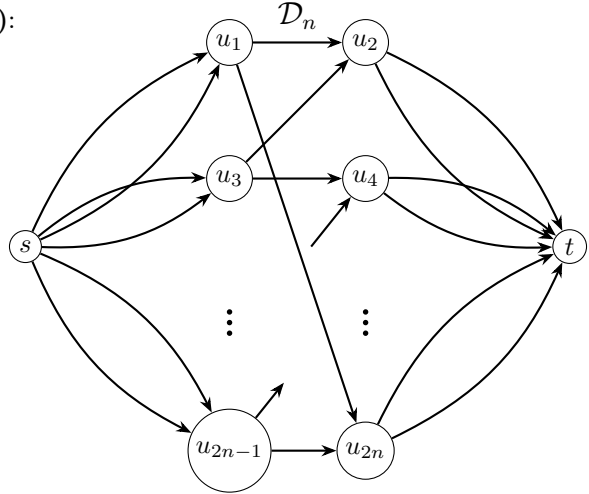
\vspace{-0.6cm}
\begin{center}
\begin{tikzpicture}[
  every node/.style={circle,draw,inner sep=2pt},
  scale=0.9
]
\node[flabel] at (0,2.4) {{\Large $\Dc_n$}};

\node (s)  at (-4,-1) {$s$};

\node (u1) at (-1, 2) {$u_1$};
\node (u2) at ( 1, 2) {$u_2$};

\node (u3) at (-1, 0) {$u_3$};
\node (u4) at ( 1, 0) {$u_4$};

\node[draw=none] at (-1,-2) {\fontsize{40}{40}\selectfont$\vdots$};
\node[draw=none] at ( 1,-2) {\fontsize{40}{40}\selectfont$\vdots$};

\node (un1) at (-1,-4) {$u_{2n-1}$};
\node (un)  at ( 1,-4) {$u_{2n}$};

\node (t) at (4,-1) {$t$};

\draw[fedge] (s) to[bend left=20] (u1);
\draw[fedge] (s) to[bend right=20] (u1);
\draw[fedge] (s) to[bend left=20] (u3);
\draw[fedge] (s) to[bend right=20] (u3);
\draw[fedge] (s) to[bend left=20] (un1);
\draw[fedge] (s) to[bend right=20] (un1);

\draw[fedge] (u2) to[bend left=20] (t);
\draw[fedge] (u2) to[bend right=20] (t);
\draw[fedge] (u4) to[bend left=20] (t);
\draw[fedge] (u4) to[bend right=20] (t);
\draw[fedge] (un) to[bend left=20] (t);
\draw[fedge] (un) to[bend right=20] (t);

\draw[fedge] (u1) -- (u2);
\draw[fedge] (u3) -- (u2);
\draw[fedge] (u3) -- (u4);

\draw[fedge] (0.2,-1) -- (u4);
\draw[fedge] (un1) -- (-0.2,-3);

\draw[fedge] (un1) -- (un);

\draw[fedge] (u1) -- (un);
\end{tikzpicture}
\end{center}

\begin{rmk}\label{rem:K5}
Note that $\Fc_{33}^{(2)}$ contains $\Fc_5$ as a butterfly minor. In fact, we may contract the idle edge $3\to 4$.
In addition, both $\Fc_{33}^{(3b)}$ and $\Fc_{33}^{(3c)}$ contain $\Fc_5$ as a butterfly minor.
In the case of $\Fc_{33}^{(3b)}$, let $e_1$ and $e_2$ be the two parallel edges from $4$ to $5$.
Delete $e_1$. Then $e_2$ becomes idle,
so we contract $e_2$, yielding $\Fc_5$.
In the case of $\Fc_{33}^{(3c)}$, delete $4 \to 6$. Then the edge $45$ becomes idle,
so we contract $45$, yielding $\Fc_5$.

In summary, $\Fc_{33}^{(2)}, \Fc_{33}^{(3b)}$ and $\Fc_{33}^{(3c)}$ contain $\Fc_5$ as a butterfly minor.
\end{rmk}

In Appendix \ref{sec:app}, we provide an alternative proof that $\Fc(\Dc_n)$ is not unimodularly equivalent to an order polytope.

\bigskip
\section{The if direction of Theorem \ref{thm:main2}}\label{sec:c-to-b}

In this section, we prove the if direction of Theorem~\ref{thm:main2} by contraposition.

\begin{lem}\label{lem:high-degree}
Assume that $G$ has no idle edges and satisfies \ref{Hvertex}.
If there exists a vertex $v\in V \setminus\{s,t\}$ with $|N^+(v) \cup N^-(v)|\geq3$,
then $G$ contains one of the forbidden minors
\[\Fc_{33}^{(4)},\quad \Fc_{33}^{(5a)},\quad \Fc_{33}^{(5b)}\] or its reversal as a butterfly minor.
\end{lem}

\begin{proof}
By \ref{Hvertex}, choose three distinct good neighbors $a, b, c$ of $v$.
Suppose that $a, b, c \in N^+(v)$.
By \ref{Hvertex}, there are edges from $s$ to $a,b,c$, and there are two parallel edges from each of $a,b,c$ to $t$. Since $v\neq s$, the vertex $v$ has an incoming edge. Moreover, since $G$ has no idle edges, we have $\operatorname{indeg}_G(v)\ge2$. Choose two directed $s$-to-$v$ paths $P_1$ and $P_2$ ending in distinct incoming edges of $v$. Let $w$ be their last common vertex before $v$.
Then the $w$-to-$v$ terminal subpaths of $P_1$ and $P_2$ are internally vertex-disjoint.
Keep one directed $s$-to-$w$ subpath contained in $P_1\cup P_2$, these two terminal subpaths, the edges from $s$ to $a,b,c$, the edges from $v$ to $a,b,c$, and the two parallel edges from each of $a,b,c$ to $t$.
Delete all other edges.
None of $a,b,c$ lies on the retained paths from $s$ to $v$, since otherwise the edge from $v$ to that vertex would create a directed cycle.
After the deletion, every internal vertex of each retained path has exactly one incoming edge and exactly one outgoing edge. Hence the retained paths can be contracted successively by butterfly contractions.
The resulting graph is $\Fc_{33}^{(4)}$.
The case $a,b,c  \in N^-(v)$ follows by reversing all arrows.

Otherwise, after reversing all arrows if necessary, we may assume that $a, b \in N^-(v)$ and $c \in N^+(v)$.
The edges guaranteed by \ref{Hvertex}, together with the edges incident with $a,b,c,s,t$, form $\Fc_{33}^{(5)}$.
Since $G$ has no idle edges, the displayed edge $45$ is not idle in $G$.
Hence, vertex $4$ has an outgoing edge $4\to x$ distinct from the displayed edge $45$.
If $x=5$, then we obtain $\Fc_{33}^{(5a)}$.
If $x=6$, then we obtain $\Fc_{33}^{(5b)}$.
Suppose that $x \notin \{5,6\}$.
Since $t=6$ is the unique sink, the edge $4\to x$ can be extended to a directed path $P$ from $4$ to $6$ whose first edge is $4\to x$.
Since each of the vertices $1,2,3$ reaches $4$ in the displayed configuration, the path $P$ cannot contain any of them; otherwise a directed cycle would arise.
If $P$ contains $5$, contract the initial subpath of $P$ from $4$ to $5$ to a single edge. Together with the displayed edge $45$, this gives two parallel edges from $4$ to $5$, and hence yields $\Fc_{33}^{(5 a)}$.
If $P$ does not contain $5$, contract all internal vertices of $P$. This produces the additional edge $46$ while preserving $5$, and hence yields $\Fc_{33}^{(5 b)}$.
\end{proof}

Thus, if $G$ avoids the forbidden minors $\Fc_{33}^{(4)},\Fc_{33}^{(5a)},\Fc_{33}^{(5b)}$, then we may assume
\begin{equation}\label{eq:deg-at-most-two}
|N^+(v) \cup N^-(v)|\leq2 \quad\text{for all }v\in V\setminus\{s,t\}.
\end{equation}

Here, we recall the notion of switch vertices.
Let $C$ be a cycle in the underlying undirected graph of a DAG.
\begin{itemize}
    \item A vertex $v$ of $C$ is \emph{source-switch} (resp. \emph{sink-switch}) if both edges of $C$ incident to $v$ are directed out of (resp. into) $v$.
    \item Write $A(C)$ for the number of source-switch vertices of $C$.
    The number of source-switch vertices is equal to the number of sink-switch vertices.
    Since the graph is acyclic, one has $A(C)>0$.
\end{itemize}

We also use the following condition.
\begin{enumerate}[label=($*$)]
\item\label{Hcycle'}
for every cycle $C$ in $G\setminus\{s,t\}$, each source-switch vertex of $C$ is adjacent to $s$ with at least two parallel edges,
and each sink-switch vertex of $C$ is adjacent to $t$ with at least two parallel edges.
\end{enumerate}

\begin{lem}\label{lem:improve}
Assume that $G$ has no idle edges and satisfies \eqref{eq:deg-at-most-two}. Then $G$ satisfies \ref{Hcycle'}.
\end{lem}
\begin{proof}
Let $p$ be a source-switch of a cycle $C$ in $ G \setminus\{s,t\}$, and let $a$ and $b$ be its two neighbors on $C$.
Then $p\to a$ and $p\to b$ are edges.
Since $p\neq s$ and $G$ has no idle edges, we have $\operatorname{indeg}_G(p)\ge2$.
If $x\to p$ satisfies $x\neq s$, then acyclicity would imply $x\notin\{a,b\}$. Hence $x,a,b$ would be three distinct neighbors of $p \in V\setminus\{s,t\}$,
contradicting \eqref{eq:deg-at-most-two}. Therefore, every edge entering $p$ comes from $s$.
Since $\operatorname{indeg}_G(p)\ge2$, there are at least two parallel edges from $s$ to $p$. The assertion for sink-switches follows by reversing all arrows.
\end{proof}

\begin{lem}\label{lem:cycle-Dn}
Assume that $G$ satisfies \ref{Hcycle'}.
If $G\setminus\{s,t\}$ contains a cycle $C$ with $A(C)=n\geq2$, then $G$ contains $\Dc_n$ as a butterfly minor.
\end{lem}

\begin{proof}
Let $p_1,q_1,p_2,q_2,\ldots,p_n,q_n$ be the source-switch vertices and sink-switch vertices of $C$ appearing in this cyclic order.
By \ref{Hcycle'}, each $p_i$ is adjacent to $s$ by at least two parallel edges, and each $q_i$ is adjacent to $t$ by at least two parallel edges.
Keep these parallel edges and the cycle $C$, delete all other edges, and contract each maximal directed subpath of $C$ between consecutive vertices.
Then the resulting graph is exactly $\Dc_n$.
\end{proof}

\begin{lem}[The $K_5$-case]\label{lem:K5case}
Assume \eqref{eq:deg-at-most-two} and \ref{Hcycle'}.
If $G\cup\{st\}$ contains a subdivision of $K_5$ and does not contain a $K_{3,3}$-subdivision, then $G$ contains $\Fc_5$ or $\Dc_n$  as a butterfly minor.
\end{lem}
\begin{proof}
Let $H$ be a $K_5$-subdivision with branch vertices $v_1,\ldots,v_5$, i.e., $v_1,\ldots,v_5$ are vertices of degree $4$.
Then \eqref{eq:deg-at-most-two} forces $s$ and $t$ to be branch vertices; otherwise some branch vertex of $G\setminus\{s,t\}$ would have degree at least three.

Let, say, $s=v_1$ and $t=v_5$. Then the paths from $s$ to $v_2,v_3,v_4$, and from $v_2,v_3,v_4$ to $t$, must be single edges.
In fact, if, say, the path from $s$ to $v_i$ has length at least two, then $v_i$ would have two neighbors among $v_2,v_3,v_4$ and one additional neighbor on the path from $s$,
giving degree at least three in $G\setminus\{s,t\}$, a contradiction to \eqref{eq:deg-at-most-two}.

Let us consider the subdivision of the triangle on the branch vertices $v_2,v_3,v_4$, which coincides with $H \setminus \{s,t\}$.
If there is a cycle with at least two source-switch vertices, then we can find $\Dc_n$ by Lemma~\ref{lem:cycle-Dn}.
Otherwise, $H \setminus \{s,t\}$ has exactly one source-switch vertex $p$ and exactly one sink-switch vertex $q$. If $p$ or $q$ lies in the interior of one of the three paths $v_2v_3,v_2v_4$ and $v_3v_4$, say, $p$ lies in $v_2v_3$,
then \ref{Hcycle'} supplies two parallel edges from $s$ to $p$.
In this case, we can find a $K_{3,3}$-subdivision by taking a bipartition $\{s,v_2,v_3\} \sqcup \{p,v_4,t\}$, a contradiction to our assumption.
Thus, all of those three paths are directed, and after relabeling their directions are assumed to be from $v_i$ to $v_j$ with $2 \leq i<j \leq 4$.
Then $v_2$ (resp. $v_4$) is a source-switch (resp. sink-switch) vertex. By \ref{Hcycle'}, there are two parallel edges from $s$ to $v_2$ and from $v_4$ to $t$.
Therefore, by contracting the three directed paths, we eventually obtain $\Fc_5$.
\end{proof}

\begin{lem}[The $K_{3,3}$-case]\label{lem:K33case}
Assume \eqref{eq:deg-at-most-two} and \ref{Hcycle'}.
If $G\cup\{st\}$ contains
a subdivision of $K_{3,3}$, then $G$ contains $\Fc_{33}^{(1)}$ or $\Fc_{33}^{(2)}$ or $\Fc_{33}^{(3)}$ or $\Dc_n$ for some $n \geq 2$ or the reversal of one of those graphs.
Moreover, if $G$ has no idle edges, then $G$ contains one of the forbidden minors in
\[\{\Fc_{33}^{(1a)},\Fc_{33}^{(1b)},\Fc_{33}^{(1c)},\Fc_{33}^{(2)},\Fc_{33}^{(3a)},\Fc_{33}^{(3b)},\Fc_{33}^{(3c)}\} \cup \{\Dc_n\}_{n \geq 2}\]
or the reversal of one of those graphs, as a butterfly minor.
\end{lem}
\begin{proof}
Let $H$ be a $K_{3,3}$-subdivision with branch vertices $v_1,\ldots,v_6$.
Then \eqref{eq:deg-at-most-two} forces $s$ and $t$ to be branch vertices of $H$ by the same argument as in Lemma~\ref{lem:K5case}.
Let, say, $s=v_1$ and $t=v_6$.
If $H\setminus\{s,t\}$ contains no cycle, then $H\setminus\{s,t\}$ is a $K_{1,3}$-subdivision.
Its center has degree three in $G\setminus\{s,t\}$, a contradiction to \eqref{eq:deg-at-most-two}.
Thus, we may assume that $H\setminus\{s,t\}$ contains a cycle $C$, and hence $s$ and $t$ belong to different bipartition classes.
By the same argument as in Lemma~\ref{lem:K5case}, each of the four branch paths from $s$ to any of $\{v_3,v_5\}$ and from $\{v_2,v_4\}$ to $t$ consists of a single edge.

Let us assume that $\{v_1, v_2, v_4\} \sqcup \{v_3,v_5,v_6\}$ is the bipartition of branch vertices of $H$.
Then $v_2$, $v_3$, $v_4$, and $v_5$ occur on $C$ in this cyclic order.
If $A(C)\geq2$, then we can find $\Dc_n$ by Lemma~\ref{lem:cycle-Dn}.
Assume $A(C)=1$, and let $p$ and $q$ be the unique source-switch and sink-switch of $C$, respectively.
By \ref{Hcycle'}, there are two parallel edges from $s$ to $p$ and two parallel edges from $q$ to $t$.

\bigskip

\begin{claim}\label{claimclaim}
    Suppose that $p$ lies in the interior of a branch path of $C$.
Then, using the two parallel edges from $s$ to $p$ guaranteed by \ref{Hcycle'}, one can replace one endpoint of that branch path by $p$ and obtain another $K_{3,3}$-subdivision whose cycle has the same unique sink-switch and whose unique source-switch is a branch vertex.
\end{claim}

\begin{proof}
Suppose that
$p$ lies in the interior of the $v_2$--$v_3$ branch path of $C$. By
\ref{Hcycle'}, there are two parallel edges from $s$ to $p$. Keep one of
these edges and delete the edge $sv_3$. The resulting graph contains a
$K_{3,3}$-subdivision with bipartition
\[
  \{s,v_2,v_4\}\sqcup\{p,v_5,t\}.
\]
Indeed, the three branch paths joining $p$ to $s$, $v_2$, and $v_4$ are,
respectively, the edge $sp$, the subpath from $p$ to $v_2$ of the old
$v_2$--$v_3$ branch path, and the concatenation of the subpath from $p$
to $v_3$ with the old $v_3$--$v_4$ branch path. Thus, the original branch
vertex $v_3$ becomes an internal vertex of the new $p$--$v_4$ branch
path, while all other branch paths remain unchanged.

In fact, the corresponding cycle is the same cycle $C$.
Hence it has the same unique sink-switch $q$, whereas
$p$ is its unique source-switch and is now a branch vertex. The other
three possible branch paths containing $p$ are treated in the same way,
after cyclically relabeling the vertices.
\end{proof}

Applying Claim~\ref{claimclaim} to $p$, and applying its reversed version to $q$, we may assume that both $p$ and $q$ are branch vertices of the $K_{3,3}$-subdivision $H$.
Since $p$ and $q$ are the only switches of $C$ and both are branch vertices, each of the four branch paths forming $C$ is directed.
Delete all edges outside $H$, except for the two parallel edges from $s$ to $p$ and the two parallel edges from $q$ to $t$ guaranteed by \ref{Hcycle'}.
Every internal vertex of such a directed branch path then has exactly one incoming edge and one outgoing edge.
Hence each directed branch path can be contracted successively by contracting idle edges.
We may therefore assume that the underlying undirected graph of $H$ is $K_{3,3}$.

Since $C$ is a $4$-cycle, either $p$ and $q$ are adjacent or they are opposite.
In the adjacent case, there are two possibilities according to which bipartition class contains $p$; in the opposite case, all possibilities are equivalent up to an automorphism of $C$ and reversal.

\bigskip

\noindent
{\bf Case 1.} (The vertices $p$ and $q$ are adjacent on $C$.)
If $p \in \{v_3, v_5\}$ and $q \in \{v_2, v_4\}$, then $G$ contains $\Fc_{33}^{(1)}$ as a butterfly minor.
If $p \in \{v_2, v_4\}$ and $q \in \{v_3, v_5\}$, then $G$ contains $\Fc_{33}^{(2)}$ as a butterfly minor.

\bigskip

\noindent
{\bf Case 2.} (The vertices $p$ and $q$ are not adjacent on $C$.)
Up to an automorphism of $C$ and reversal of all arrows, we may assume that
$p= v_2$ and $q=v_4$.
Then $G$ contains $\Fc_{33}^{(3)}$ as a butterfly minor.

\bigskip

Moreover, by the assumption that there is no idle edge in $G$, we can find minimal versions without idle edges for $\Fc_{33}^{(1)}$ and $ \Fc_{33}^{(3)}$.
Relabel the resulting six vertices as $1,\ldots,6$ in accordance with the notation in Section~\ref{sec:onlyif}.

\begin{itemize}
    \item[(1)] $\Fc_{33}^{(1)}$:
There must exist an edge $x \to 3$ distinct from the displayed edge $2 \to 3$,
and an edge $4 \to y$ distinct from the displayed edge $4 \to 5$ to make $23$ and $45$ non-idle, respectively.
By \eqref{eq:deg-at-most-two}, $x \in \{1,2\}$ and $y \in \{5,6\}$.
Then $(x,y) = (1,6), (2,6), (2,5), (1,5)$
correspond to $\Fc_{33}^{(1a)}$, $\Fc_{33}^{(1b)}$, $\Fc_{33}^{(1c)}$, and $\Fc_{33}^{(1b)}$.

    \item[(2)] $\Fc_{33}^{(3)}$:
There must exist an edge $3 \to x$ distinct from the displayed edge $3 \to 5$ and an edge $4 \to y$
distinct from the displayed edge $4 \to 5$ to make $35$ and $45$ non-idle, respectively.
By \eqref{eq:deg-at-most-two}, $x , y \in \{5,6\}$.
Then $(x,y) = (5,5), (6,5), (6,6), (5,6)$
correspond to $\Fc_{33}^{(3a)}$, $\Fc_{33}^{(3b)}$, $\Fc_{33}^{(3c)}$, and $\Fc_{33}^{(3b)}$.

\end{itemize}
Thus, we obtain the desired conclusion.
\end{proof}

We now complete the proof by contraposition.
Suppose that $G$ is not $st$-planar.
Then $G\cup\{st\}$ contains a subdivision of $K_5$ or $K_{3,3}$.
If some internal vertex has at least three internal neighbors, Lemma~\ref{lem:high-degree} gives one of $\Fc_{33}^{(4)},\Fc_{33}^{(5 a)},\Fc_{33}^{(5 b)}$ or its reversal.
We may therefore assume that \eqref{eq:deg-at-most-two} holds.
By Lemma~\ref{lem:improve}, $G$ satisfies \ref{Hcycle'}.
We can now apply Lemmas~\ref{lem:K5case} and \ref{lem:K33case}, according as $G\cup\{st\}$ contains a subdivision of $K_5$ or $K_{3,3}$.
In either case, we obtain one of the remaining forbidden minors, possibly one of $\Fc_{33}^{(2)},\Fc_{33}^{(3 b)},\Fc_{33}^{(3 c)}$.
By Remark~\ref{rem:K5}, each of the latter three contains $\Fc_5$ as a butterfly minor. Therefore, $G$ contains one of the forbidden DAGs listed in Theorem~\ref{thm:main2} or its reversal.
\qed

\begin{rmk}\label{rmk:forbidden}
We finally point out that the three-good-neighbor condition \ref{Hvertex} is essential for obtaining the list in Section~\ref{sec:onlyif}.
If this condition is omitted, then additional forbidden butterfly minors appear.
For example, the two graphs in Figure~\ref{fig:extra-forbidden} have this property.
These graphs are obtained from $\Dc_2$ by replacing some of the arrows.
\begin{figure}[h]
\centering
\begin{tikzpicture}[scale=0.7, every node/.style={fvertex}]
\node (1) at (-2,0) {1};
\node (2) at (0,1.5) {2};
\node (3) at (0,-1.5) {3};
\node (4) at (2,1.5) {4};
\node (5) at (2,-1.5) {5};
\node (6) at (4,0) {6};

\draw[fedge] (1) to[bend left=15] (2);
\draw[fedge] (1) to[bend right=15] (2);
\draw[fedge] (1) -- (3);
\draw[fedge] (2) -- (3);
\draw[fedge] (2) -- (4);
\draw[fedge] (2) -- (5);
\draw[fedge] (3) -- (4);
\draw[fedge] (3) -- (5);
\draw[fedge] (4) to[bend left=15] (6);
\draw[fedge] (4) to[bend right=15] (6);
\draw[fedge] (5) to[bend left=15] (6);
\draw[fedge] (5) to[bend right=15] (6);
\end{tikzpicture}
\qquad
\begin{tikzpicture}[scale=0.7, every node/.style={fvertex}]
\node (1) at (-2,0) {1};
\node (2) at (0,1.5) {2};
\node (3) at (0,-1.5) {3};
\node (4) at (2,1.5) {4};
\node (5) at (2,-1.5) {5};
\node (6) at (4,0) {6};

\draw[fedge] (1) to[bend left=15] (2);
\draw[fedge] (1) to[bend right=15] (2);
\draw[fedge] (1) -- (3);
\draw[fedge] (2) -- (3);
\draw[fedge] (2) -- (4);
\draw[fedge] (2) -- (5);
\draw[fedge] (3) -- (4);
\draw[fedge] (3) -- (5);
\draw[fedge] (4) -- (5);
\draw[fedge] (4) -- (6);
\draw[fedge] (5) to[bend left=15] (6);
\draw[fedge] (5) to[bend right=15] (6);
\end{tikzpicture}
\caption{Examples of additional forbidden graphs which do not satisfy the three-good-neighbor condition \ref{Hvertex}.}
\label{fig:extra-forbidden}
\end{figure}
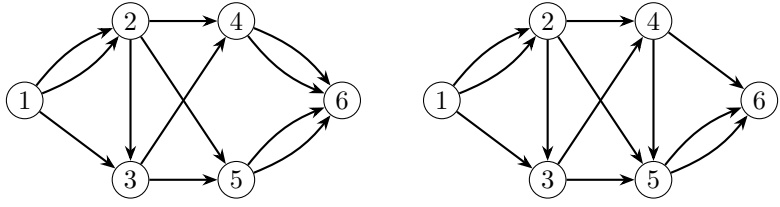

Similarly, for each $n\geq 3$, one can construct additional forbidden butterfly minors from $\Dc_n$ which do not arise from any $\Dc_{n'}$ with $n'<n$.
This suggests that, without the three-good-neighbor condition, a complete list of forbidden butterfly minors in our setting would be considerably more complicated.
\end{rmk}

\bigskip
\appendix
\section{An alternative proof for the family $\Dc_n$}\label{sec:app}

In this appendix, we give an alternative proof that $\mathcal F(\Dc_n)$ is not unimodularly equivalent to any order polytope.

\begin{lem}\label{lem:ineq}
Let $n\ge2$ and let $\ab=(a_1,\ldots,a_{2n-1})\in\ZZ_{>0}^{2n-1}$ satisfy $\sum_{k=1}^{2n-1}a_k=4n-1$. Then
\begin{align}
& a_1+a_1a_2+\cdots+a_{2n-2}a_{2n-1}+a_{2n-1}
 +\sum_{k=1}^{2n-1}(2^{a_k}-1)+1
\geq 14n. \label{eq:14n}
\end{align}
The equality holds if and only if
\[\ab=(3,1,3,1,\ldots,3,1,3).\]
\end{lem}
\begin{proof}
Set $b_k=a_k-1\geq 0$.  Then $\sum_{k=1}^{2n-1} b_k=2n$. We first compute
\begin{align*}
a_1+a_1a_2+\cdots+a_{2n-2}a_{2n-1}+a_{2n-1}
&=b_1+b_{2n-1}+2+
\sum_{k=1}^{2n-2}(b_k+1)(b_{k+1}+1)\\
&=2(b_1+\cdots+b_{2n-1})+2n+
\sum_{k=1}^{2n-2}b_k b_{k+1}\\
&=6n+\sum_{k=1}^{2n-2}b_k b_{k+1}.
\end{align*}
Thus, it is enough to show that, for $(b_1,\ldots,b_{2n-1}) \in \ZZ_{\geq 0}^{2n-1}$ with $\sum b_k=2n$,
\[f_n(b_1,\ldots,b_{2n-1}):=\sum_{k=1}^{2n-2}b_k b_{k+1}+\sum_{k=1}^{2n-1}2^{b_k+1} \geq 10n-2,\]
and the equality holds if and only if
\[(b_1,\ldots,b_{2n-1})=(2,0,2,0,\ldots,2,0,2).\]
The equality for this case is immediate.

We prove this inequality by induction on $n$. For $n=2$, a direct check of the fifteen triples satisfying $b_1+b_2+b_3=4$ gives
\[b_1b_2+b_2b_3+2^{b_1+1}+2^{b_2+1}+2^{b_3+1}\geq18,\]
with the equality only if $(b_1,b_2,b_3)=(2,0,2)$.
In fact, since $b_1$ and $b_3$ are symmetric, it suffices to check the following nine representatives of the fifteen triples:
\[(4,0,0),(0,4,0),(3,1,0),(3,0,1),(0,3,1),(2,2,0),(2,0,2),(2,1,1),(1,2,1).\]

Assume $n > 2$ and choose a minimizer. By reversing $b_1,\ldots,b_k$ into $b_k,\ldots,b_1$ if necessary, we may assume that $b_1\geq b_k$ for each $k$. Indeed, for $b_1<b_k$,
\[f_n(b_1,\ldots,b_k,b_{k+1},\ldots)-f_n(b_k,\ldots,b_1,b_{k+1},\ldots)=(b_k-b_1)b_{k+1}\geq0,\]
where $b_{2n}=0$. Hence, $b_1\geq2$; otherwise $2n=\sum_{k=1}^{2n-1} b_k \leq (2n-1)b_1 \leq 2n-1$, a contradiction.
Similarly, after fixing $b_1$, we may assume that $b_2$ is minimal among $b_2,\ldots,b_{2n-1}$. Indeed, for $b_2>b_k$,
\[f_n(b_1,b_2,\ldots,b_k,b_{k+1},\ldots)-f_n(b_1,b_k,\ldots,b_2,b_{k+1},\ldots)=(b_2-b_k)(b_1-b_{k+1})\geq0.\]
Hence, $b_2\leq1$; otherwise $2n=\sum_{k=1}^{2n-1} b_k \geq (2n-1)b_2 \geq 4n-2$, a contradiction.
If $b_2=1$, then $b_k \ge 1$ for all $k \ge 2$.
Since $b_1 \ge 2$ and $\sum_{k=1}^{2n-1}b_k=2n$, we have $(b_1,b_2,\ldots,b_{2n-1})=(2,1,\ldots,1)$ and then
\[f_n(b_1,\dots,b_{2n-1})=10n-1 > 10n -2 =f_n(2,0,2,0,\dots,2,0,2),\]
a contradiction.
Thus we have
$b_2=0$.
If $b_1 \ge 3$, then
\[
f_n(b_1,0,b_3,\ldots,b_{2n-1})-f_n(b_1-1,1,b_3,\ldots,b_{2n-1})=2^{b_1}-b_1-b_3-1
\ge 2^{b_1}-2b_1-1 >0,
\]
a contradiction.
Hence $b_1=2$.

With $b_1=2$ and $b_2=0$, we have
\[f_n(b_1,\ldots,b_{2n-1})=f_n(2,0,b_3,\ldots,b_{2n-1})=10+f_{n-1}(b_3,\ldots,b_{2n-1}),\]
and $b_3+\cdots +b_{2n-1} = 2 (n-1)$.
By the induction hypothesis,
\[
f_{n-1}(b_3,\ldots,b_{2n-1}) \ge 10(n-1)-2
\]
with the equality if and only if \[(b_3,\ldots,b_{2n-1})=(2,0,2,0,\ldots,2,0,2).\]
Thus we have
\[f_n(b_1,\ldots,b_{2n-1})=
10+f_{n-1}(b_3,\ldots,b_{2n-1})
\ge 10+ 10(n-1)-2=10n-2
,\]
and
with the equality if and only if \[(b_1,\ldots,b_{2n-1})=(2,0,2,0,\ldots,2,0,2),\]
as desired.

\end{proof}

Let $A_m$ denote the antichain with $m$ elements.

\begin{lem}\label{lem:Dn-order-candidate}
If $\Fc(\Dc_n)$ is unimodularly equivalent to the order polytope $\Oc(P)$
of a poset $P$, then $P$ is isomorphic to the poset
\[P_n=A_3\oplu A_1\oplu A_3\oplu A_1\oplu\cdots\oplu A_1\oplu A_3,\]
where $P_n$ has $n$ copies of $A_3$ and $n-1$ copies of $A_1$.
\end{lem}

\begin{proof}
We compute the dimension $d$, the number $N$ of vertices, the number $M$ of facets, and the codegree $r$ for $\Fc(\Dc_n)$. Note that $|V|=2n+2$ and $|E|=6n$:
\[ d=6n-(2n+2)+1=4n-1; \;\; N=2n \cdot 2 \cdot 2=8n; \;\; M=6n; \;\; r=2n. \]
Since $r=2n$, the maximum number of cover relations in a chain of $P$ should be $2n-2$.
Moreover, $\Fc(\Dc_n)$ is Gorenstein, since every vertex $v \in V \setminus \{s,t\}$ of $\Dc_n$ has the same number of
incoming and outgoing edges.
See \cite[Proposition~6.15]{vonBell}.
It is known that $\Oc(P)$ is Gorenstein if and only if $P$ is graded (\cite{Hibi}).
Thus, $P$ has $2n-1$ nonempty rank levels $P_1, \dots, P_{2n-1}$.
Let $a_k = |P_k|$. Then
\[\sum_{k=1}^{2n-1}a_k=|P|=4n-1.\]
Since every cover relation of $P$ joins two consecutive rank levels, the Hasse diagram of $P$ is a spanning subgraph of that of
\[P(a_1,\ldots,a_{2n-1}):=A_{a_1} \oplus \cdots \oplus A_{a_{2n-1}}.\]
Regarding the order polytope of $P(a_1,\ldots,a_{2n-1})$, we observe the following:
\begin{itemize}
    \item The number of facets of $\Oc(P(a_1,\ldots,a_{2n-1}))$ is the number of cover relations of $P(a_1,\ldots,a_{2n-1}) \cup \{\hat{0},\hat{1}\}$, which is $a_1+a_1a_2+\cdots+a_{2n-2}a_{2n-1}+a_{2n-1}$;
    \item The number of vertices of $\Oc(P(a_1,\ldots,a_{2n-1}))$ is the number of antichains, which is $(2^{a_1}-1)+\cdots+(2^{a_{2n-1}}-1)+1$.
\end{itemize}
Given a poset $Q$, let $\mathcal{M}(Q) = H(Q) + {\rm Ant}(Q)$,
where $H(Q)$ is the number of cover relations in $Q \cup \{\hat{0}, \hat{1}\}$ and ${\rm Ant}(Q)$ is the number of antichains in $Q$.
Since $\mathcal O(P)$ and $\mathcal F(\Dc_n)$ have the same numbers of facets and vertices, we have \[ \mathcal{M}(P)=H(P)+\operatorname{Ant}(P)=6n+8n=14n. \]
By Lemma~\ref{lem:ineq},
\[
\mathcal{M}(P(a_1,\ldots,a_{2n-1})) \geq 14n
\] with the equality if and only if $(a_1,\ldots,a_{2n-1})=(3,1,3,1,\ldots,3,1,3)$.
If $e=ij$ is an edge of the Hasse diagram of $P(a_1,\ldots,a_{2n-1})$ and not an edge of the Hasse diagram of $P$, then $\{i,j\}$ is an antichain of $P$ and not an antichain of $P(a_1,\ldots,a_{2n-1})$.
Thus, if $H(P(a_1,\ldots,a_{2n-1})) - H(P) =\alpha$, then
\begin{eqnarray*}
\mathcal{M}(P(a_1,\ldots,a_{2n-1}))
&=&H(P(a_1,\ldots,a_{2n-1})) + {\rm Ant}(P(a_1,\ldots,a_{2n-1}))\\
&\le& (H(P) + \alpha) +  ({\rm Ant}(P) -\alpha) \\
&=& \mathcal{M}(P).
\end{eqnarray*}
Hence, if $(a_1,\ldots,a_{2n-1}) \neq (3,1,3,1,\ldots,3,1,3)$, then $\mathcal{M}(P) > 14n$,
contradicting $\mathcal{M}(P)=14n$.
On the other hand,
if the Hasse diagram of $P$ is a proper subgraph of the Hasse diagram of
$P(3,1,3,1,\ldots,3,1,3)$,
then $P$ is not graded,
implying that $\Oc(P)$ is not Gorenstein.
Indeed, every rank level of size three is adjacent to a singleton rank level. If one of the corresponding cover relations is deleted,
then the affected element belongs to a maximal chain strictly shorter than a maximal chain passing through one of the remaining elements. Hence the resulting poset is not graded.

Therefore, the only possible poset is $P(3,1,3,1,\ldots,3,1,3)$ itself.
\end{proof}

\begin{prop}\label{prop:Dn-not-order}
For any $n\geq 2$, the flow polytope $\Fc(\Dc_n)$ is not unimodularly equivalent to any order polytope.
\end{prop}

\begin{proof}
By Lemma~\ref{lem:Dn-order-candidate}, it remains to compare the normalized volume $V_n$ of $\Fc(\Dc_n)$ with the normalized volume $V_n'$ of $\Oc(P_n)$ where $P_n$ is the only possible candidate $A_3 \oplus A_1 \oplus \cdots \oplus A_3$.

We first compute the normalized volume of $\Fc(\Dc_n)$.
By \cite[Theorem~1.7]{BBDH}, the normalized volume of the flow polytope of the extension of a bipartite graph is equal to the number of matchings of the associated almost-degree-whiskered graph.
In the present case, the underlying bipartite graph is the cycle $C_{2n}$, and attaching one leaf to each vertex of $C_{2n}$ produces the $2n$-sunlet graph.
Hence, the normalized volume of $\Fc(\Dc_n)$ coincides with the number of matchings in the $2n$-sunlet graph.
Let $M_m$ denote the number of matchings in the $m$-sunlet graph.
From \cite[Lemma 2.2]{DNP}, we have $M_m =L_m(2)$, where $L_m(x)$ is the Lucas polynomial.
Since $L_0(2) = L_1(2)= 2$ and $L_m(2) = 2L_{m-1}(2) + L_{m-2}(2)$, it follows that $L_m(2)$ coincides with the Pell-Lucas number
\[
(1+\sqrt{2})^m+(1-\sqrt{2})^m.\]
Taking $m=2n$, we obtain \[V_n=(1+\sqrt{2})^{2n}+(1-\sqrt{2})^{2n}.\]

Next, we compute the normalized volume of the order polytope of
\[P_n=A_3\oplu A_1\oplu A_3\oplu A_1\oplu\cdots\oplu A_1\oplu A_3,\]
where there are $n$ copies of $A_3$ and $n-1$ copies of $A_1$.
By Stanley's theorem on order polytopes, the normalized volume of $\Oc(P_n)$ is the number $e(P_n)$ of linear extensions of $P_n$; see \cite[Theorem~4.1]{Stanley}.
Since $P_n$ is an ordinal sum, every element of an earlier summand must appear before every element of a later summand
in any linear extension.
Since $e(A_1) =1$ and $e(A_3)=3!=6$, we have
\[V_n' =e(P_n)=e(A_3)e(A_1)e(A_3)e(A_1) \cdots e(A_1)e(A_3) =6^n.\]

We now compare these two volumes. Let $\alpha = (1+\sqrt{2})^2= 3+2\sqrt{2}$ and $\beta = (1-\sqrt{2})^2 =  3-2\sqrt{2}$.
Since both $\alpha$ and $\beta$ are positive,
\[V_n = \alpha^n + \beta^n< (\alpha + \beta)^n=6^n=V_n'\] for $n \ge 2$.
Thus, the normalized volumes are different, and the two polytopes cannot be unimodularly equivalent.
\end{proof}

\end{document}